\documentclass[12pt]{amsart}

\usepackage{mathrsfs, amssymb,amsthm, amsfonts, amsmath}
\usepackage[all]{xy}
\usepackage{upgreek}
\usepackage{amsmath}

\newtheorem{theorem}[equation]{Theorem}
\newtheorem*{theorem*}{Theorem}
\newtheorem{lemma}[equation]{Lemma}
\newtheorem{corollary}[equation]{Corollary}

\newtheorem{question}[equation]{Question}
\newtheorem{proposition}[equation]{Proposition}

\theoremstyle{definition}
\newtheorem{example}[equation]{Example}

\newtheorem*{definition*}{Definition}

\theoremstyle{remark}
\newtheorem{remark}[equation]{Remark}

\makeatletter\@addtoreset{equation}{section}
\makeatother

\newcommand{\QQ}{\mathbb{Q}}
\newcommand{\ZZ}{\mathbb{Z}}
\newcommand{\PP}{\mathbb{P}}
\newcommand{\KK}{\mathbb{K}}

\newcommand{\LLL}{{\mathscr{L}}}
\newcommand{\HHH}{{\mathscr{H}}}
\newcommand{\SSS}{\mathfrak{S}}

\newcommand{\mumu}{{\boldsymbol{\mu}}}
\newcommand{\op}{\mathrm{op}}

\newcommand{\Aut}{\operatorname{Aut}}
\newcommand{\Bir}{\operatorname{Bir}}
\newcommand{\PGL}{\operatorname{PGL}}
\newcommand{\SL}{\operatorname{SL}}
\newcommand{\Gal}{\operatorname{Gal}}
\newcommand{\Br}{\operatorname{Br}}
\newcommand{\Fix}{\operatorname{Fix}}
\newcommand{\z}{\operatorname{z}}
\newcommand{\rk}{\operatorname{rk}}
\newcommand{\mult}{\operatorname{mult}}
\newcommand{\Pic}{\operatorname{Pic}}
\newcommand{\rkPic}{\rk\Pic}
\newcommand{\WE}{\mathrm{W}(\mathrm{E}_6)}

\def \ge {\geqslant}
\def \le {\leqslant}

\date{}

\title{Birational automorphisms of Severi--Brauer surfaces}

\author{Constantin Shramov}

\address{Steklov Mathematical Institute of Russian Academy of Sciences, 8 Gubkina st.,
Moscow, 119991, Russia
\newline
National Research University Higher School of Economics, Laboratory of Algebraic Geometry, 6 Usacheva str., Moscow, 119048, Russia
}

\email{costya.shramov@gmail.com}

\begin{document}

\begin{abstract}
We prove that a finite group acting by birational automorphisms of a non-trivial Severi--Brauer surface
over a field of characteristic zero contains a normal abelian subgroup of index at most~$3$.
Also, we find an explicit bound for orders of such finite groups in the case when the base field contains all roots of~$1$.
\end{abstract}

\maketitle

\section{Introduction}

A \emph{Severi--Brauer variety} of dimension $n$ over a field $\KK$ is a variety
that becomes isomorphic to the projective space of dimension $n$ over the algebraic closure of~$\KK$.
Such varieties are in one-to-one correspondence with central simple algebras of dimension~$(n+1)^2$
over~$\KK$. They have many nice geometric properties. For instance, it is known that
a Severi--Brauer variety over $\KK$
is isomorphic to the projective space if and only if it has a $\KK$-point.
We refer the reader to \cite{Artin} and~\cite{Kollar-SB} for other basic
facts concerning Severi--Brauer varieties.

Automorphism groups of Severi--Brauer varieties can be described in terms
of the corresponding central simple algebras, see Theorem~E on page~266
of~\cite{Chatelet}, or~\mbox{\cite[\S1.6.1]{Artin}}, or \cite[Lemma~4.1]{ShramovVologodsky}.
As for the group of birational automorphisms, something is known in the case of surfaces.
Namely, let
$\KK$ be a field of characteristic zero (more generally, one can assume that either $\KK$ is perfect, or
its characteristic is different from $2$ and~$3$). Let $S$ be a \emph{non-trivial} Severi--Brauer surface over $\KK$,
i.e. one that is not isomorphic to~$\PP^2$.
In this case generators for the group $\Bir(S)$ of birational automorphisms of $S$ are known,
see \cite{Weinstein}, or~\cite{Weinstein-new}, or~\cite[Theorem~2.6]{Is-UMN},
or Theorem~\ref{theorem:Weinstein} below.
Moreover, relations between the generators are known as well, see~\mbox{\cite[\S3]{IskovskikhTregub}}.
This may be thought of as an analog of the classical theorem of Noether describing the generators
of the group $\Bir(\PP^2)$ over an algebraically closed field,
and the results concerning relations between them  (see \cite{Gizatullin}, \cite{Iskovskikh-SimpleGiz}, \cite{IKT}).

Regarding finite groups acting by automorphisms or birational automorphisms on Severi--Brauer surfaces, the following result is known.

\begin{theorem}[{see \cite[Proposition~1.9(ii),(iii)]{ShramovVologodsky}, \cite[Corollary~1.5]{ShramovVologodsky}}]
\label{theorem:ShramovVologodsky}
Let $S$ be a non-trivial Severi--Brauer surface over a field $\KK$ of characteristic zero.
Suppose that $\KK$ contains all roots of $1$.
The following assertions hold.
\begin{itemize}
\item[(i)] If $G\subset\Aut(S)$ is a finite subgroup,
then every non-trivial element of $G$ has order~$3$, and
$G$ is a $3$-group of order at most~$27$.

\item[(ii)] There exists a constant $B=B(S)$ such that for any finite subgroup $G\subset\Bir(S)$
one has $|G|\le B$.
\end{itemize}
\end{theorem}

In this paper we prove the result making Theorem~\ref{theorem:ShramovVologodsky} more precise.

\begin{theorem}\label{theorem:main}
Let $S$ be a non-trivial Severi--Brauer surface over a field $\KK$ of characteristic zero,
and let $G\subset\Bir(S)$ be a finite group.
The following assertions hold.
\begin{itemize}
\item[(i)] The order of $G$ is odd.

\item[(ii)] The group $G$ is either abelian, or contains a normal abelian subgroup
of index~$3$.

\item[(iii)] If $\KK$ contains all roots of $1$, then
$G$ is an abelian $3$-group of order at most~$27$.
\end{itemize}
\end{theorem}

I do not know if the bounds in Theorem~\ref{theorem:main}(ii),(iii) are optimal.
In particular, I am not aware of an example of a finite non-abelian group
acting by birational (or biregular, cf.~Proposition~\ref{proposition:SB-Bir-vs-Aut} below)
automorphisms on a non-trivial Severi--Brauer surface.
Note that
it is easy to construct an example of a non-trivial Severi--Brauer
surface with an action of a group of order~$9$, see Example~\ref{example:cyclic-algebra} below.
In certain cases this is the largest finite subgroup of the automorphism
group of a Severi--Brauer surface; see Lemma~\ref{lemma:27}, which improves the result
of Theorem~\ref{theorem:ShramovVologodsky}(i).
It would be interesting to obtain a complete description of finite groups
acting by biregular and birational automorphisms
on Severi--Brauer surfaces, cf.~\cite{DI}.

Theorem~\ref{theorem:main}(ii) can be reformulated
by saying that the \emph{Jordan constant}
(see e.g.~\mbox{\cite[Definition~1.1]{Yasinsky}} for a definition)
of the birational automorphism group of a
non-trivial Severi--Brauer surface over a field $\KK$ of characteristic zero is at most $3$.
This shows one of the amazing differences between birational geometry of
non-trivial Severi--Brauer surfaces and the projective plane,
since in the latter case the corresponding Jordan constant may be much larger.
For instance, if the base field is algebraically closed, the Jordan constant
of the group of birational automorphisms of $\PP^2$ equals~$7200$,
see~\mbox{\cite[Theorem~1.9]{Yasinsky}}. Moreover, by the remark made after
Theorem~5.3 in~\cite{Serre2009} the multiplicative analog of this constant
for $\PP^2$
equals~\mbox{$2^{10}\cdot 3^4\cdot 5^2\cdot 7$} in the case of the algebraically closed field
of characteristic zero, while by Theorem~\ref{theorem:main}(ii)
a similar constant for a non-trivial Severi--Brauer surface
also equals either~$1$ or~$3$.

To prove Theorem~\ref{theorem:main}, we establish the following intermediate result that
might be of independent interest (see also Proposition~\ref{proposition:summary} below for a more precise
statement).

\begin{proposition}\label{proposition:SB-Bir-vs-Aut}
Let $S$ be a non-trivial Severi--Brauer surface over a field of characteristic zero,
and let $G\subset\Bir(S)$ be a finite non-abelian group.
Then $G$ is conjugate to a subgroup of $\Aut(S)$.
\end{proposition}

It would be interesting to find out if Proposition~\ref{proposition:SB-Bir-vs-Aut}
holds for finite abelian subgroups of birational automorphism groups
of non-trivial Severi--Brauer surfaces.

\smallskip
The plan of the paper is as follows. In~\S\ref{section:birational} we study surfaces that may be birational to
a non-trivial Severi--Brauer surface.
In~\S\ref{section:G-birational} we study finite groups acting on such surfaces.
In~\S\ref{section:bounds} we prove Proposition~\ref{proposition:SB-Bir-vs-Aut} and Theorem~\ref{theorem:main}.

\smallskip
\textbf{Notation and conventions.}
Throughout the paper assume that all varieties are projective.
We denote by $\mumu_n$ the cyclic group of order $n$,
and by $\SSS_n$ the symmetric group on $n$ letters.

Given a field $\KK$, we denote by $\bar{\KK}$ its algebraic closure.
For a variety $X$ defined over $\KK$, we denote by $X_{\bar{\KK}}$
its scalar extension to~$\bar{\KK}$.
By a point of degree $d$ on a variety defined over some field $\KK$ we mean a closed point whose
residue field is an extension of~$\KK$ of degree~$d$; a $\KK$-point is a point of degree~$1$.

By a linear system on a variety~$X$ over a field $\KK$ we mean a twisted linear subvariety
of a linear system on $X_{\bar{\KK}}$ defined over~$\KK$; thus, a linear system is not
a projective space in general, but a Severi--Brauer variety.

By a degree of a subvariety $Z$ of a Severi--Brauer variety $X$ over a field $\KK$
we mean the degree of the subvariety $Z_{\bar{\KK}}$ of $X_{\bar{\KK}}\cong\PP^n_{\bar{\KK}}$ with respect to
the hyperplane in~$\PP^n_{\bar{\KK}}$.

For a Severi--Brauer variety
$X$ corresponding to the central simple algebra $A$,
we denote by~$X^{\op}$ the Severi--Brauer variety corresponding to the algebra opposite to~$A$.

A del Pezzo surface is a smooth surface with an ample anticanonical class.
For a del Pezzo surface $S$, by its degree we mean its (anti)canonical degree~$K_S^2$.

Let $X$ be a variety over an algebraically closed field with an action of
a group $G$, and let $g$ be an element of $G$. By $\Fix_X(G)$ and $\Fix_X(g)$ we denote
the loci of fixed points of the group $G$ and the element $g$ on $X$, respectively.

\smallskip
\textbf{Acknowledgements.}
I am grateful to  A.\,Trepalin and V.\,Vologodsky for many useful discussions.
I was partially supported by
the HSE University Basic Research Program,
Russian Academic Excellence Project~\mbox{``5-100''},
by the Young Russian Mathematics award, and by the Foundation for the
Advancement of Theoretical Physics and Mathematics ``BASIS''.

\section{Birational models of Severi--Brauer surfaces}
\label{section:birational}

In this section we study surfaces that may be birational to
a non-trivial Severi--Brauer surface.

The following general result is sometimes referred to as the theorem of Lang and Nishimura.

\begin{theorem}[{see e.g.~\cite[Proposition~IV.6.2]{Kollar-RatCurves}}]
\label{theorem:Lang-Nishimura}
Let $X$ and $Y$ be smooth projective varieties over an arbitrary field $\KK$.
Suppose that $X$ is birational to $Y$. Then $X$ has a $\KK$-point if and only if $Y$ has a $\KK$-point.
\end{theorem}

\begin{corollary}\label{corollary:Lang-Nishimura}
Let $X$ and $Y$ be smooth projective varieties over an arbitrary field $\KK$, and let $r$ be a positive integer.
Suppose that $X$ is birational to $Y$, and $X$ has a point of degree not divisible by~$r$.
Then $Y$ has a point of degree not divisible by~$r$.
\end{corollary}

The following result concerning Severi--Brauer surfaces is well-known.

\begin{theorem}[{see e.g.~\cite[Theorem~53(2)]{Kollar-SB}}]
\label{theorem:SB-point-degree}
Let $S$ be a non-trivial Severi--Brauer surface over an arbitrary field. Then $S$ does not contain points of degree $d$ not divisible by $3$.
\end{theorem}

\begin{corollary}\label{corollary:many-dP}
Let $S$ be a non-trivial Severi--Brauer surface over an arbitrary field. Then
$S$ is not birational to any conic bundle, and not birational to any del Pezzo surface of degree
different from $3$, $6$, and $9$.
\end{corollary}

\begin{proof}
Suppose that $S$ is birational to a surface $S'$ with a conic bundle structure~\mbox{$\phi\colon S'\to C$}.
Then $C$ is a conic itself, so that $C$ has a point of degree $2$. This implies that
the surface $S'$ has a point of degree $2$ or $4$, and by Corollary~\ref{corollary:Lang-Nishimura}
the surface~$S$ has a point of degree not divisible by~$3$.
This gives a contradiction with Theorem~\ref{theorem:SB-point-degree}.

Now suppose that $S$ is birational to a del Pezzo surface $S'$
of degree $d$ not divisible by $3$. Then the intersection of two general elements
of the anticanonical linear system~$|-K_{S'}|$ is an effective zero-cycle of degree $d$
defined over the base field. Thus $S'$ has a point of degree not divisible by $3$.
By Corollary~\ref{corollary:Lang-Nishimura}  this again gives a contradiction with Theorem~\ref{theorem:SB-point-degree}.
\end{proof}

\begin{corollary}\label{corollary:dP-large-Picard-rank}
Let $S$ be a non-trivial Severi--Brauer surface over an arbitrary field. Then~$S$
is not birational to any del Pezzo surface of degree $6$ of Picard rank greater than~$2$,
and not birational to any del Pezzo surface of degree $3$ of Picard rank greater than~$3$.
\end{corollary}

\begin{proof}
Suppose that $S$ is birational to a del Pezzo surface $S'$ as above.
Then there exists a birational contraction from $S'$ to a del Pezzo surface
$S''$ of degree $K_{S''}^2>K_{S'}^2$ and Picard rank~\mbox{$\rkPic(S'')=\rkPic(S')-1$}.

If $K_{S'}^2=6$ and $\rkPic(S')>2$, then $\rkPic(S'')>1$, and thus~\mbox{$6<K_{S''}^2<9$}.
This gives a contradiction with Corollary~\ref{corollary:many-dP}.

If $K_{S'}^2=3$ and $\rkPic(S')>3$, then~$S''$ is either a del Pezzo surface of degree~$6$
with~\mbox{$\rkPic(S'')>2$}, which is impossible by the above argument, or a del Pezzo surface of degree
not divisible by $3$, which is impossible by Corollary~\ref{corollary:many-dP}.
\end{proof}

To proceed we will need the following general fact about non-trivial Severi--Brauer surfaces.

\begin{lemma}\label{lemma:SB-surface-curves}
Let $S$ be a non-trivial Severi--Brauer surface over an arbitrary field $\KK$, and let~$C$ be a curve on $S$.
Then the degree of $C$ is divisible by $3$.
\end{lemma}

\begin{proof}
It is well-known (see for instance \cite[Exercise~3.3.5(iii)]{GS})
that there exists an exact sequence of groups
$$
1\to\Pic(S)\to\Pic(S_{\bar{\KK}})^{\Gal(\bar{\KK}/\KK)}\stackrel{b}\to \Br(\KK)_3,
$$
where $\Br(\KK)_3$ is the $3$-torsion part of the Brauer group of $\KK$. Furthermore, the image
of~$b$ is non-trivial since the Severi--Brauer surface
$S$ is non-trivial, see for instance \cite[Exercise~3.3.4]{GS}.
This means that any line bundle on $S$ has degree divisible by $3$, and the assertion follows.
\end{proof}

\begin{remark}
For an alternative proof of Lemma~\ref{lemma:SB-surface-curves},
one can consider the image $C'$ of~$C$ under a general
automorphism of $S$
(see \cite[Lemma~4.1]{ShramovVologodsky} for a description of the automorphism group of~$S$).
Then the zero-cycle $Z=C\cap C'$ is defined over $\KK$ and has degree coprime to~$3$,
so that the assertion follows from Theorem~\ref{theorem:SB-point-degree}.
\end{remark}

Given $d\le 6$ distinct points $P_1,\ldots,P_d$ on $\PP^2_{\bar{\KK}}$ over an algebraically closed field $\bar{\KK}$, we will say that
they are \emph{in general position} if no three of them are contained in a line,
and all six (in the case $d=6$) are not
contained in a conic (cf.~\mbox{\cite[Remark~IV.4.4]{Manin}}). If~$P$ is a point of degree $d$ on a Severi--Brauer
surface $S$ over a perfect field~$\KK$, we will say that~$P$ is in general position if the $d$ points
of the set~\mbox{$P_{\bar{\KK}}\subset\PP^2_{\bar{\KK}}$} are in general position. Note that if $P$ is a point
of degree~$d$ in general
position, then the blow up of $S$ at
$P$ is a del Pezzo surface of degree $9-d$, see
for instance~\mbox{\cite[Theorem~IV.2.6]{Manin}}.

\begin{lemma}\label{lemma:SB-general-position}
Let $S$ be a non-trivial Severi--Brauer surface over a perfect field~$\KK$,
and let~\mbox{$P\in S$} be a point of degree $d$. Suppose that $d=3$ or $d=6$. Then
$P$ is in general position.
\end{lemma}

\begin{proof}
Suppose that $d=3$ and $P$ is not in general position. Then the three points
of $P_{\bar{\KK}}$ are contained in some line $L$ on $\PP^2_{\bar{\KK}}$.
The line $L$ is $\Gal(\bar{\KK}/\KK)$-invariant
and thus defined over $\KK$,
which is impossible by Lemma~\ref{lemma:SB-surface-curves}.

Now suppose that $d=6$ and $P$ is not in general position. Let $P_{\bar{\KK}}=\{P_1,\ldots,P_6\}$.
If at least four of the points~\mbox{$P_1,\ldots,P_6$} are contained in a line, then a line with such a property is unique,
and we again get a contradiction with Lemma~\ref{lemma:SB-surface-curves}.

Assume that no four of the points~\mbox{$P_1,\ldots,P_6$} are contained in a line, but some three of them (say,
the points~$P_1$, $P_2$, and $P_3$) are contained in a line which we denote by $L$. The line $L$ is not $\Gal(\bar{\KK}/\KK)$-invariant by Lemma~\ref{lemma:SB-surface-curves},
so that there exists a line $L'\neq L$ that is $\Gal(\bar{\KK}/\KK)$-conjugate to~$L$.
If $L'$ does not pass through any of the points $P_1$, $P_2$, and~$P_3$, then $L$ and $L'$ are the only
lines in $\PP^2_{\bar{\KK}}$ that contain three of the points $P_1,\ldots,P_6$. This gives a contradiction
with Lemma~\ref{lemma:SB-surface-curves}. Hence, up to relabelling the points, we may assume that $P_3=L\cap L'$,
and the points $P_4$ and $P_5$ are contained in~$L'$.
Let $L_{ij}$ be the line passing through the points $P_i$ and $P_j$, where $i\in\{1,2\}$ and $j\in\{4,5\}$.
Note that $L_{ij}$ are pairwise different, none of them coincides with $L$ or $L'$, and no three of them
intersect at one point.
If the point $P_6$ is not contained in any of the lines $L_{ij}$, then
$L$ and $L'$ are the only
lines that contain three of the points $P_1,\ldots,P_6$, which is impossible by Lemma~\ref{lemma:SB-surface-curves}.
If $P_6$ is an intersection point of two of the lines $L_{ij}$, then
there are exactly four
lines that contain three of the points $P_1,\ldots,P_6$, which is also impossible by Lemma~\ref{lemma:SB-surface-curves}.
Thus, we see that $P_6$ must be contained in a unique line among $L_{ij}$, say, in $L_{15}$. Now there are exactly three lines
that contain three of the points $P_1,\ldots,P_6$, namely, $L$, $L'$, and $L_{15}$.
We see that each of the points $P_1$, $P_3$, and $P_5$ is contained in two of these lines,
while each of the points $P_2$, $P_4$, and $P_6$ is contained in a unique such line. However, the Galois group~\mbox{$\Gal(\bar{\KK}/\KK)$} acts
transitively on the points~\mbox{$P_1,\ldots, P_6$}, which gives a contradiction.

Therefore, we may assume that the points $P_1,\ldots, P_6$ are contained in some irreducible conic~$C$.
Obviously, such a conic is unique, and thus $\Gal(\bar{\KK}/\KK)$-invariant. This again gives a contradiction
with Lemma~\ref{lemma:SB-surface-curves}.
\end{proof}

Let $S$ be a non-trivial Severi--Brauer surface over a perfect field~$\KK$, and let~$A$
be the corresponding central simple algebra.
Recall that $S^{\op}$ the Severi--Brauer surface corresponding to the algebra opposite to~$A$.
There are two classes of interesting birational maps from $S$ to $S^{\op}$ which we describe below
(cf. \cite[Lemma~3.1]{IskovskikhTregub}, or  cases~(c) and~(e) of~\cite[Theorem~2.6(ii)]{Is-UMN}).

Let $P$ be a point of degree $3$ on $S$. Then $P$ is in general position by Lemma~\ref{lemma:SB-general-position}.
Blowing up $P$ and blowing down the proper transforms of the three lines on $S_{\bar{\KK}}\cong\PP^2_{\bar{\KK}}$
passing through the pairs of points of $P_{\bar{\KK}}$,
we obtain a birational map $\tau_P$ to another Severi--Brauer surface $S'$.
This map is given by the linear system of conics passing through $P$.

Similarly, let $P$ be a point of degree $6$ on $S$. Then $P$ is in general position by Lemma~\ref{lemma:SB-general-position}.
Blowing up $P$ and blowing down the proper transforms of the six conics on~\mbox{$S_{\bar{\KK}}\cong\PP^2_{\bar{\KK}}$}
passing through the quintuples of points of $P_{\bar{\KK}}$,
we obtain a birational map $\eta_P$ to another Severi--Brauer surface $S'$.
This map is given by the linear system of quintics singular at~$P$ (or, in other words, at each point of~$P_{\bar{\KK}}$).

In both of the above cases one has $S'\cong S^{\op}$. This follows from a well-known general
fact about the degrees of birational maps between Severi--Brauer varieties with given classes,
see for instance~\cite[Exercise~3.3.7(iii)]{GS}.
For more details on the maps $\tau_P$, see~\mbox{\cite[\S2]{Corn}}.

\begin{remark}\label{remark:Amitsur}
It is known that $S$ and $S^{\op}$ are the only Severi--Brauer surfaces birational to~$S$, see for instance~\cite[Exercise~3.3.6(v)]{GS}.
\end{remark}

The following theorem was first published in~\cite{Weinstein}
(see also~\cite{Weinstein-new}). It can be also obtained
as a particular case of a much more general result, see~\cite[Theorem~2.6]{Is-UMN}.
We provide a sketch of its proof for the reader's convenience.

\begin{theorem}\label{theorem:Weinstein}
Let $S$ be a non-trivial Severi--Brauer surface over a perfect field $\KK$, and let $S'$ be
a del Pezzo surface over $\KK$ with $\rkPic(S')=1$.
Suppose that $S$ is birational to~$S'$. Then either
$S'\cong S$, or $S'\cong S^{\op}$. Moreover, any birational map~\mbox{$\theta\colon S\dasharrow S'$}
can be written as a composition
$$
\theta=\theta_1\circ\ldots\circ\theta_k,
$$
where each of the maps $\theta_i$ is either an automorphism, or a map $\tau_P$ or $\eta_P$ for some point~$P$ of
$S$ or~$S^{\op}$.
\end{theorem}

\begin{proof}[Sketch of the proof]
Let $\theta\colon S\dasharrow S'$ be a birational map, and suppose that $\theta$ is not
an isomorphism.
Choose a very ample linear system $\LLL'$ on $S'$, and let $\LLL$ be its proper transform
on $S$. Then $\LLL$ is a mobile non-empty (and in general incomplete)
linear system on $S$. Write
$$
\LLL\sim_{\QQ} -\mu K_S
$$
for some positive rational number $\mu$; note that $3\mu$ is an integer.
By the Noether--Fano inequality
(see \cite[Lemma~1.3(i)]{Is-UMN}),
one has $\mult_{P}(\LLL)>\mu$ for some point $P$ on $S$. Let~$d$ be the degree of $P$, and let
$L_1$ and $L_2$ be two general members of the linear system~$\LLL$ (defined over~$\bar{\KK}$).
We see that
$$
9\mu^2=L_1\cdot L_2\ge d\mult_{P}(\LLL)^2>d\mu^2,
$$
and thus $d<9$. Hence one has $d=3$ or $d=6$ by Theorem~\ref{theorem:SB-point-degree},
and the point $P$ is in general position by Lemma~\ref{lemma:SB-general-position}.

Consider a birational map $\theta_P$ defined as follows: if $d=3$, we let $\theta_P=\tau_P$,
and if~\mbox{$d=6$}, we let $\theta_P=\eta_P$. Let $\theta^{(1)}=\theta\circ\theta_P^{-1}$. Let
$\LLL_1$ be the proper transform of $\LLL$ (or $\LLL'$) on the
surface~$S^{\op}$, and write
$$
\LLL_1\sim_{\QQ} -\mu_1 K_{S^{\op}}
$$
for some positive rational number $\mu_1$ such that~\mbox{$3\mu_1\in\ZZ$}.
Using the information about~$\theta_P$ provided in~\cite[Lemma~3.1]{IskovskikhTregub}, we see that
$\mu_1<\mu$.
Therefore, applying the same procedure to the surface~\mbox{$S_1=S^{\op}$}, the birational map $\theta^{(1)}$,
and the linear system $\LLL_1$ and arguing by induction, we prove the theorem.
\end{proof}

A particular case of Theorem~\ref{theorem:Weinstein} is the following result that we will need below.

\begin{corollary}\label{corollary:Weinstein}
Let $S$ be a non-trivial Severi--Brauer surface over a perfect field.
Then~$S$ is not birational to any del Pezzo surface $S'$ of degree~$3$ or~$6$
with~\mbox{$\rkPic(S')=1$}.
\end{corollary}

The part of Corollary~\ref{corollary:Weinstein} concerning del Pezzo surfaces of degree $3$ also follows
from~\mbox{\cite[Chapter~V]{Manin}}.
The part concerning del Pezzo surfaces of degree $6$ can be obtained from
\cite[\S2]{IskovskikhTregub}.

Corollaries~\ref{corollary:dP-large-Picard-rank} and \ref{corollary:Weinstein} show that a del Pezzo surface
of degree $6$ birational to a non-trivial Severi--Brauer
surface must have Picard rank equal to~$2$, and a del Pezzo surface
of degree $3$ birational to a non-trivial Severi--Brauer
surface must have Picard rank equal to~$2$ or~$3$.
In the next section we will obtain further restrictions on such surfaces provided
that they are $G$-minimal with respect to some finite group~$G$.

\section{$G$-birational models of Severi--Brauer surfaces}
\label{section:G-birational}

In this section we study finite groups acting on surfaces birational
to a non-trivial Severi--Brauer surface.

We start with del Pezzo surfaces of degree $6$. Recall that over an algebraically closed field $\bar{\KK}$
of characteristic zero
a del Pezzo surface of degree $6$ is unique up to isomorphism, and its
automorphism group is isomorphic to $(\bar{\KK}^*)^2\rtimes (\SSS_3\times\mumu_2)$.
More details on this can be found in~\cite[Theorem~8.4.2]{Dolgachev}.

Given a del Pezzo surface $S'$ of degree $6$ over an arbitrary field $\KK$ of
characteristic zero, we will call $\SSS_3\times\mumu_2$ its \emph{Weyl group}.
For every element $\theta\in\Aut(S')$ we will refer to the image of $\theta$
under the composition of the embedding $\Aut(S')\hookrightarrow\Aut(S'_{\bar{\KK}})$ with the natural homomorphism
$$
\Aut(S'_{\bar{\KK}})\to\SSS_3\times\mumu_2
$$
as the image of $\theta$ in the Weyl group. Similarly, we will consider the image of the Galois group
$\Gal(\bar{\KK}/\KK)$ in the Weyl group.

Let $\sigma\colon\PP^2\dasharrow\PP^2$ be the \emph{standard Cremona involution},
that is, a birational involution acting as
$$
(x:y:z)\mapsto\left(\frac{1}{x}:\frac{1}{y}:\frac{1}{z}\right)
$$
in some homogeneous coordinates $x$, $y$, and $z$. This involution becomes regular
on the del Pezzo surface of degree $6$ obtained as a blow up of $\PP^2$ at the points
$(1:0:0)$, $(0:1:0)$, and $(0:0:1)$. Let $\hat{\sigma}$ be its image in the Weyl group $\SSS_3\times\mumu_2$.
Then $\hat{\sigma}$ is the generator of the center $\mumu_2$ of the Weyl group.
If one thinks about~\mbox{$\SSS_3\times\mumu_2$}
as the group of symmetries of a regular hexagon, then $\hat{\sigma}$ is an involution
that interchanges the opposite sides of the hexagon or, in other words, a rotation by~$180^\circ$.
If $S'$ is a del Pezzo surface of degree $6$ over some field of characteristic zero and $\theta$ is its automorphism,
we will say that $\theta$ is \emph{of Cremona type} if its image in the Weyl group
coincides with $\hat{\sigma}$.

\begin{lemma}\label{lemma:Cremona-type}
Let $S'$ be a del Pezzo surface of degree $6$ over a field $\KK$ of characteristic zero, and let $\theta$
be its automorphism of Cremona type. Then $\theta$ is an involution that has exactly four fixed points on~$S'_{\bar{\KK}}$.
\end{lemma}

\begin{proof}
Using a birational contraction $S'_{\bar{\KK}}\to\PP^2_{\bar{\KK}}$,
consider the automorphism $\theta$
as a birational automorphism of $\PP^2_{\bar{\KK}}$.
Then $\theta$ can be represented as a composition of the standard Cremona
involution $\sigma$ with some element of the standard torus acting on $\PP^2_{\bar{\KK}}$.
Thus, one can choose homogeneous coordinates $x$, $y$, and $z$
on $\PP^2_{\bar{\KK}}$ so that $\theta$ acts as
$$
(x:y:z)\mapsto\left(\frac{\alpha}{x}:\frac{\beta}{y}:\frac{\gamma}{z}\right)
$$
for some non-zero $\alpha,\beta,\gamma\in\bar{\KK}$.
Therefore, $\theta$ is conjugate to the involution~$\sigma$ via the automorphism
$$
(x:y:z)\mapsto (\sqrt{\alpha}x:\sqrt{\beta}y:\sqrt{\gamma}z).
$$
This shows that $\theta$ has the same number of fixed points on $\PP^2_{\bar{\KK}}$ as $\sigma$,
while it is easy to see that the fixed points of the latter are the four points
$$
(1:1:1),\quad (-1:1:1), \quad (1:-1:1),\quad (1:1:-1).
$$
It remains to notice that a fixed point of $\theta$ on $S'_{\bar{\KK}}$ cannot be contained
in a $(-1)$-curve, and thus all of them are mapped to fixed points of $\theta$ on~$\PP^2_{\bar{\KK}}$.
\end{proof}

\begin{lemma}\label{lemma:dP6-rk-1}
Let $S$ be a non-trivial Severi--Brauer surface over a field $\KK$ of characteristic zero, and let
$S'$ be a del Pezzo surface of degree $6$ over $\KK$. Suppose that there exists a finite group
$G\subset\Aut(S')$ such that $\rkPic(S')^G=1$. Then $S'$ is not birational to $S$.
\end{lemma}

\begin{proof}
Suppose that $S'$ is birational to $S$.
We know from Corollary~\ref{corollary:dP-large-Picard-rank} that~\mbox{$\rkPic(S')\le 2$}.
Hence by Corollary~\ref{corollary:Weinstein} one has~\mbox{$\rkPic(S')=2$}.
Since $S'$ does not have a structure of a conic bundle by Corollary~\ref{corollary:many-dP}, we see that
$S'$ has a contraction on a del Pezzo surface of larger degree.
Again by Corollary~\ref{corollary:many-dP},
this means that $S'$ is a blow up of a Severi--Brauer surface at a point of degree~$3$.
Hence the image $\Gamma$ of the Galois group~\mbox{$\Gal(\bar{\KK}/\KK)$} in the Weyl group~\mbox{$\SSS_3\times\mumu_2$}
contains an element of order~$3$.

Since the cone of effective curves
on $S'$ has two extremal rays and $\rkPic(S')^G=1$, we see that $G$ must contain an element
whose image in the Weyl group has order~$2$. On the other hand,
the image of $G$ in the Weyl group commutes with~$\Gamma$. Since $\hat{\sigma}$
is the only element of order $2$ in~\mbox{$\SSS_3\times\mumu_2$}  that commutes with an element
of order~$3$, we conclude that $G$
must contain an element $\theta$
of Cremona type. The element~$\theta$ has exactly $4$ fixed points on $S'_{\bar{\KK}}$ by Lemma~\ref{lemma:Cremona-type}.
This implies that there is a point of degree not divisible by~$3$ on $S'$.
Thus the assertion follows from Corollary~\ref{corollary:Lang-Nishimura}  and Theorem~\ref{theorem:SB-point-degree}.
\end{proof}

Now we deal with del Pezzo surfaces of degree $3$, i.e. smooth cubic surfaces in $\PP^3$. Recall from that for a del Pezzo surface
$S'$ of degree $3$ over a field $\KK$ the action of the groups~\mbox{$\Aut(S')$} and $\Gal(\bar{\KK}/\KK)$ on $(-1)$-curves on $S'$ defines homomorphisms
of these groups to the Weyl group~$\WE$. Furthermore, the homomorphism
$\Aut(S')\to\WE$ is an embedding. The order of the Weyl group~$\WE$ equals~$2^7\cdot 3^4\cdot 5$.
We refer the reader to~\cite[Theorem~8.2.40]{Dolgachev}
and~\cite[Chapter~IV]{Manin} for details.

\begin{lemma}\label{lemma:dP3-not-3-group}
Let $S$ be a non-trivial Severi--Brauer surface over a field $\KK$ of characteristic zero, and let
$S'$ be a del Pezzo surface of degree $3$ over $\KK$. Suppose that there exists a non-trivial automorphism
$g\in\Aut(S')$ such that the order of $g$ is not a power of $3$. Then~$S'$ is not birational to $S$.
\end{lemma}

\begin{proof}
We may assume that the order $p=\mathrm{ord}(g)$ is prime. Thus one has $p=2$ or~\mbox{$p=5$}, since the order of the Weyl group
$\WE$ is not divisible by primes greater than~$5$.
The action of $g$ on $S'_{\bar{\KK}}$ can be of one of the three types listed in
\cite[Table~2]{Trepalin-cubic}; in the notation of \cite[Table~2]{Trepalin-cubic} these are
types~1, 2, and~6. It is straightforward to check
that if $g$ is of type $1$, then the fixed point locus
$\Fix_{S'_{\bar{\KK}}}(g)$ consists of a smooth elliptic curve and one isolated point;
if $g$ is of type $2$, then $\Fix_{S'_{\bar{\KK}}}(g)$ consists of a $(-1)$-curve and three isolated points;
if $g$ is of type $6$, then $\Fix_{S'_{\bar{\KK}}}(g)$ consists of a four isolated points. Note that a $\Gal(\bar{\KK}/\KK)$-invariant
$(-1)$-curve on $S'$ always contains a $\KK$-point, since such a curve is a line
in the anticanonical embedding of $S'$. Therefore, in each of the above three cases we
find a $\Gal(\bar{\KK}/\KK)$-invariant set of points on $S'_{\bar{\KK}}$ of cardinality coprime to~$3$.
Thus the assertion follows from Corollary~\ref{corollary:Lang-Nishimura}  and Theorem~\ref{theorem:SB-point-degree}.
\end{proof}

\begin{corollary}\label{corollary:dP3-rkPic-3}
Let $S$ be a non-trivial Severi--Brauer surface over a field $\KK$ of characteristic zero, and let
$S'$ be a del Pezzo surface of degree $3$ over $\KK$ birational to $S$.
Suppose that there exists a subgroup $G\subset\Aut(S')$ such that
$\rkPic(S')^G=1$. Then~$\rkPic(S')=3$.
\end{corollary}

\begin{proof}
We know from Corollaries~\ref{corollary:dP-large-Picard-rank} and
\ref{corollary:Weinstein} that
either $\rkPic(S')=2$, or $\rkPic(S')=3$.
Suppose that $\rkPic(S')=2$, so that the cone of effective curves
on $S'$ has two extremal rays. Since $\rkPic(S')^G=1$, the group $G$ must contain an element of even order.
Therefore, the assertion follows from Lemma~\ref{lemma:dP3-not-3-group}.
\end{proof}

\begin{corollary}\label{corollary:dP3-3-group}
Let $S$ be a non-trivial Severi--Brauer surface over a field $\KK$ of characteristic zero, and let
$S'$ be a del Pezzo surface of degree $3$ over $\KK$ birational to $S$.
Let~\mbox{$G\subset\Aut(S')$} be a subgroup such that $\rkPic(S')^G=1$. Then $G$ is isomorphic to
a subgroup of~$\mumu_3^3$.
\end{corollary}

\begin{proof}
We know from Corollary~\ref{corollary:dP3-rkPic-3} that $\rkPic(S')=3$.
Hence $S'$ is a blow up of a Severi--Brauer surface at two points of degree $3$.
This means that the image $\Gamma$ of the Galois group~\mbox{$\Gal(\bar{\KK}/\KK)$} in
the Weyl group $\WE$ contains an element conjugate to
$$
\gamma=(123)(456)
$$
in the notation of~\cite[\S4]{Trepalin-cubic}.
We may assume that $\Gamma$ contains $\gamma$ itself, so that
the image of~$G$ in $\WE$ is contained in the centralizer $Z(\gamma)$
of $\gamma$. By~\mbox{\cite[Proposition~4.5]{Trepalin-cubic}}
one has
$$
Z(\gamma)\cong(\mumu_3^2\rtimes\mumu_2)\times \SSS_3.
$$
On the other hand, we know from Lemma~\ref{lemma:dP3-not-3-group} that the order of $G$ is a power of $3$.
Since the Sylow $3$-subgroup of $Z(\gamma)$ is isomorphic to $\mumu_3^3$, the required
assertion follows.
\end{proof}

\begin{remark}
For an alternative proof of Corollary~\ref{corollary:dP3-3-group}, suppose that $G$ is a non-abelian $3$-group.
One can notice that
in this case the image of $\Gal(\bar{\KK}/\KK)$ in $\WE$ is
either trivial, or is generated by an element of type~$\mathrm{A}_2^3$ in the notation of
\cite{Carter}. In the former case $\rkPic(S')=7$, which is impossible by
Corollary~\ref{corollary:dP-large-Picard-rank}.
In the latter case $\rkPic(S')=1$, which is impossible by Corollary~\ref{corollary:Weinstein}.
Thus, $G$ is an abelian $3$-group. The rest is more or less straightforward, cf.~\mbox{\cite[Theorem~6.14]{DI}}.
\end{remark}

Let us summarize the results of this section.

\begin{proposition}\label{proposition:summary}
Let $S$ be a non-trivial Severi--Brauer surface over a field $\KK$ of characteristic zero, and let
$G\subset\Bir(S)$ be a finite subgroup.
Then $G$ is conjugate either to a subgroup of
$\Aut(S)$, or to a subgroup of
$\Aut(S^{op})$, or to a subgroup of
$\Aut(S')$, where $S'$ is a del Pezzo surface of degree $3$ over $\KK$ birational to $S$
such that~\mbox{$\rkPic(S')=3$} and~\mbox{$\rkPic(S')^G=1$}.
In the latter case $G$ is isomorphic to
a subgroup of~$\mumu_3^3$.
\end{proposition}

\begin{proof}
Regularizing the action of $G$ and running a $G$-Minimal Model Program (see~\mbox{\cite[Theorem~1G]{Iskovskikh80}}),
we obtain a $G$-surface
$S'$ birational to $S$, such that~$S'$ is either a del Pezzo surface with
$\rkPic(S')^G=1$, or a conic bundle. The case of a conic bundle is impossible by Corollary~\ref{corollary:many-dP}.
Thus by Corollary~\ref{corollary:many-dP} and Lemma~\ref{lemma:dP6-rk-1}
we conclude that $S'$ is a del Pezzo surface of degree $9$ or $3$.
In the former case $S'$ is a Severi--Brauer surface itself, so
that $S'$ is isomorphic either to $S$ or to $S^{\op}$ by Remark~\ref{remark:Amitsur}.
In the latter case
we have $\rkPic(S')=3$ by Corollary~\ref{corollary:dP3-rkPic-3}. Furthermore, in this case~$G$ is isomorphic to
a subgroup of~$\mumu_3^3$ by
Corollary~\ref{corollary:dP3-3-group}.
\end{proof}

I do not know the answer to the following question.

\begin{question}
Does there exist an example of a finite abelian group $G$
acting on a smooth cubic surface~$S'$ over a field of characteristic zero,
such that $S'$ is birational to a non-trivial
Severi--Brauer surface and~\mbox{$\rkPic(S')^G=1$}?
In other words, does there exist a non-trivial Severi--Brauer surface
$S$ over a field of characteristic zero and a finite abelian
group~\mbox{$G\subset\Bir(S)$}, such that
the action of $G$ can be regularized on
some smooth cubic surface, but
$G$ is not conjugate to a subgroup of~\mbox{$\Aut(S)$}?
\end{question}

\section{Automorphisms of Severi--Brauer surfaces}
\label{section:bounds}

In this section we prove
Proposition~\ref{proposition:SB-Bir-vs-Aut}
and Theorem~\ref{theorem:main}.
We start with a couple of simple auxiliary results.

Let $\HHH_3$ denote the Heisenberg group of order $27$; this is the only non-abelian
group of order $27$ and exponent $3$. Its center $\z(\HHH_3)$ is isomorphic to $\mumu_3$,
and there is a non-split exact sequence
$$
1\to\z(\HHH_3)\to\HHH_3\to\mumu_3^2\to 1.
$$
On the other hand, one can also represent $\HHH_3$ as a semi-direct product $\HHH_3\cong\mumu_3^2\rtimes\mumu_3$.

\begin{lemma}\label{lemma:P2}
Let $\bar{\KK}$ be an algebraically closed field of characteristic zero,
and let
$$
G\subset\Aut(\PP^2)\cong\PGL_3(\bar{\KK})
$$
be a finite subgroup.
The following assertions hold.
\begin{itemize}
\item[(i)] If the order of $G$ is odd, then $G$ is either abelian, or contains a normal abelian subgroup
of index~$3$.

\item[(ii)] If the order of $G$ is odd and $G$ is non-trivial,
then $G$ contains a subgroup $H$ such that either~\mbox{$|\Fix_{\PP^2}(H)|=3$},
or $\Fix_{\PP^2}(H)$ has a unique isolated point.
Moreover, one can choose $H$ with such a property so that either $H=G$,
or $H$ is a normal subgroup of index $3$ in~$G$.

\item[(iii)] If $G\cong\HHH_3$, then $G$ contains an element $g$ such that
$\Fix_{\PP^2}(g)$ has a unique isolated point.

\item[(iv)] The group $G$ is not isomorphic to $\mumu_3^3$.
\end{itemize}
\end{lemma}

\begin{proof}
Let $\tilde{G}$ be the preimage of $G$ under the natural projection
$$
\pi\colon\SL_3(\bar{\KK})\to\PGL_3(\bar{\KK}),
$$
and let $V\cong\bar{\KK}^3$ be the corresponding three-dimensional representation of $\tilde{G}$.
We will use  the classification of finite subgroups
of $\PGL_3(\bar{\KK})$, see for instance~\cite[Chapter~V]{Blichfeldt}.

Suppose that the order of $G$ is odd. Then it follows from the classification that either~$\tilde{G}$ is abelian, so that $V$ splits as a sum of three one-dimensional
$\tilde{G}$-representations such that not all of them are isomorphic to each other; or $\tilde{G}$ is non-abelian and there exists a
surjective homomorphism $\tilde{G}\to\mumu_3$ whose kernel $\tilde{H}$ is an abelian group,
so that $V$ splits as a sum of three one-dimensional
$\tilde{H}$-representations, and $\tilde{G}/\tilde{H}\cong\mumu_3$ transitively permutes these
$\tilde{H}$-representations. In other words, the group $G$ cannot be primitive (see~\mbox{\cite[\S60]{Blichfeldt}}
for the terminology), and $V$ cannot split as a sum of a one-dimensional and an irreducible two-dimensional
$\tilde{G}$-representation.
This proves assertion~(i).

If $\tilde{G}$ is abelian, let $\tilde{H}=\tilde{G}$. Thus,
if $|G|$ is odd and $G$ is non-trivial,
in all possible cases we see that the group
$H=\pi(\tilde{H})\subset\PGL_3(\bar{\KK})$ is a group with the properties required
in assertion~(ii).

Suppose that $G\cong\HHH_3$.
Then it follows from the classification (cf.~\cite[6.4]{Borel}) that
$$
\tilde{G}\cong\mumu_3^3\rtimes\mumu_3,
$$
and $\tilde{H}\cong\mumu_3^3$.
The elements of $\tilde{H}$ can be simultaneously diagonalized.
Hence the image~\mbox{$H\cong\mumu_3^2$} of $\tilde{H}$ in $\PGL_3(\bar{\KK})$
contains an element $g$
such that $\Fix_{\PP^2}(h)$ consists of a line and an isolated point.
This proves assertion~(iii).

Assertion~(iv) directly follows from the classification (cf.~\cite[6.4]{Borel}).
\end{proof}

Most of our remaining arguments are based on the following observation.

\begin{lemma}
\label{lemma:SB-Aut}
Let $S$ be a non-trivial Severi--Brauer surface over a field of characteristic zero, and
let $G\subset\Aut(S)$ be a finite subgroup. Then the order of $G$ is odd.
\end{lemma}

\begin{proof}
Suppose that the order of $G$ is even. Then $G$ contains an element $g$ of order $2$.
Consider the action of $g$ on $S_{\bar{\KK}}\cong\PP^2_{\bar{\KK}}$.
The fixed point locus $\Fix_{S_{\bar{\KK}}}(g)$ is a union of a line and a unique isolated point $P$. Since the
Galois group $\Gal(\bar{\KK}/\KK)$ commutes with $g$, the point $P$ is $\Gal(\bar{\KK}/\KK)$-invariant,
which is impossible by assumption.
\end{proof}

\begin{corollary}
\label{corollary:SB-Aut}
Let $S$ be a non-trivial Severi--Brauer surface over a field of characteristic zero, and
let $G\subset\Aut(S)$ be a finite subgroup. Then
\begin{itemize}
\item[(i)] the group $G$ is either abelian, or contains a normal abelian subgroup
of index~$3$;

\item[(ii)] there exists a $G$-invariant point of degree $3$ on $S$.
\end{itemize}
\end{corollary}

\begin{proof}
By Lemma~\ref{lemma:SB-Aut}, the order of $G$ is odd.
The action of $G$ on $S_{\bar{\KK}}\cong\PP^2_{\bar{\KK}}$ gives
an embedding $G\subset\PGL_3(\bar{\KK})$.
By Lemma~\ref{lemma:P2}(i)
the group $G$ is either abelian,  or contains a normal abelian subgroup
of index~$3$. This proves assertion~(i).

To prove assertion~(ii), we may assume that $G$ is non-trivial.
By Lemma~\ref{lemma:P2}(ii), the group $G$ contains a subgroup $H$ such that either $|\Fix_{S_{\bar{\KK}}}(H)|=3$,
or $\Fix_{S_{\bar{\KK}}}(H)$ has a unique isolated point; moreover, one can choose $H$ with such a property so that either~$H$ coincides with~$G$,
or $H$ is a normal subgroup of index $3$ in~$G$.
In any case, $\Fix_{S_{\bar{\KK}}}(H)$ cannot have a unique isolated point, because this
point would be $\Gal(\bar{\KK}/\KK)$-invariant,
which is impossible by assumption.
Hence $|\Fix_{S_{\bar{\KK}}}(H)|=3$. Since $H$ is a normal subgroup in $G$, the set $\Fix_{S_{\bar{\KK}}}(H)$
is $G$-invariant. This gives a $G$-invariant point of degree $3$ on~$S$ and proves assertion~(ii).
\end{proof}

\begin{remark}
It is interesting to note that the analogs of Lemma~\ref{lemma:SB-Aut} and Corollary~\ref{corollary:SB-Aut}
do not hold for Severi--Brauer curves, i.e. for conics. Thus, a conic over the field
$\mathbb{R}$ of real numbers defined by the equation
$$
x^2+y^2+z^2=0
$$
in $\PP^2$ with homogeneous coordinates $x$, $y$, and $z$ has no $\mathbb{R}$-points. However, it
is acted on by all finite groups embeddable into $\PGL_2(\mathbb{C})$,
that is, by cyclic groups, dihedral groups, the tetrahedral group~$\mathfrak{A}_4$, the octahedral group~$\mathfrak{S}_4$, and the icosahedral group~$\mathfrak{A}_5$.
From this point of view finite groups acting on non-trivial Severi--Brauer curves
are \emph{more} complicated than those acting on non-trivial Severi--Brauer surfaces.
It would be interesting to obtain a complete classification of finite groups acting on Severi--Brauer surfaces
similarly to what is done for conics in~\cite{Garcia-Armas}.
\end{remark}

Recall from~\S\ref{section:birational} that a  Severi--Brauer
surface $S$ is birational to the surface~$S^{op}$. In particular,
the groups $\Bir(S)$ and $\Bir(S^{op})$ are (non-canonically) isomorphic,
and the group $\Aut(S^{op})$ is (non-canonically) realized as a subgroup
of~\mbox{$\Bir(S)$}. Note also that~\mbox{$\Aut(S^{op})\cong \Aut(S)$},
although these groups are not conjugate in~\mbox{$\Bir(S)$}.

Corollary~\ref{corollary:SB-Aut} has the following geometric consequence.

\begin{corollary}\label{corollary:conjugate}
Let $S$ be a non-trivial Severi--Brauer surface over a field of characteristic zero, and
let $G\subset\Aut(S^{op})$ be a finite subgroup. Then $G$ is conjugate to a subgroup of~$\Aut(S)$.
\end{corollary}

\begin{proof}
By Corollary~\ref{corollary:SB-Aut}(ii) the group $G$ has
an invariant point $P$ of degree $3$ on $S'$, and by Lemma~\ref{lemma:SB-general-position} this point is in general position.
Blowing up~$P$ and blowing down the proper transforms of the three lines on $S'_{\bar{\KK}}\cong\PP^2_{\bar{\KK}}$
passing through the pairs of the three points of $P_{\bar{\KK}}$,
we obtain a (regular) action of $G$ on the surface~$S$ together with a $G$-equivariant
birational map $\tau_P\colon S'\dasharrow S$, cf.~\S\ref{section:birational}.
This means that $G$ is conjugate to a subgroup of~$\Aut(S)$.
\end{proof}

Now we prove Proposition~\ref{proposition:SB-Bir-vs-Aut}.

\begin{proof}[Proof of Proposition~\ref{proposition:SB-Bir-vs-Aut}]
We know from Proposition~\ref{proposition:summary} that $G$ is conjugate to a subgroup of
$\Aut(S')$, where $S'\cong S$ or $S'\cong S^{op}$. In the former case we are done.
In the latter case~$G$ is conjugate to a subgroup of $\Aut(S)$ by Corollary~\ref{corollary:conjugate}.
\end{proof}

Similarly to Lemma~\ref{lemma:SB-Aut}, we prove the following.

\begin{lemma}
\label{lemma:27}
Let $\KK$ be a field of characteristic zero that contains all roots of $1$.
Let $S$ be a non-trivial Severi--Brauer surface over $\KK$, and
let $G\subset\Aut(S)$ be a finite subgroup.
Then $G$ is isomorphic to a subgroup of~$\mumu_3^2$.
\end{lemma}
\begin{proof}
We know from Theorem~\ref{theorem:ShramovVologodsky}(i)
that every non-trivial element of $G$ has order $3$, and~\mbox{$|G|\le 27$}.
Assume that $G$ is not isomorphic to a subgroup of~$\mumu_3^2$.
Then either $G\cong\mumu_3^3$ or~\mbox{$G\cong\HHH_3$}.
The former case is impossible by Lemma~\ref{lemma:P2}(iv).
Thus, we have $G\cong\HHH_3$.
By Lemma~\ref{lemma:P2}(iii)
the group $G$ contains an element $g$ such that $\Fix_{S_{\bar{\KK}}}(g)$ has a unique isolated point. This latter point
must be $\Gal(\bar{\KK}/\KK)$-invariant, which is impossible by assumption.
\end{proof}

Finally, we prove our main result.

\begin{proof}[Proof of Theorem~\ref{theorem:main}]
Assertion (i) follows from Proposition~\ref{proposition:summary}
and Lemma~\ref{lemma:SB-Aut}.
Assertion~(ii) follows from Proposition~\ref{proposition:SB-Bir-vs-Aut} and
Corollary~\ref{corollary:SB-Aut}(i).
Assertion (iii) follows
from Proposition~\ref{proposition:summary}
and Lemma~\ref{lemma:27}.
\end{proof}

I do not know if the bound provided by Theorem~\ref{theorem:main}(iii)
(or Lemma~\ref{lemma:27}) is optimal. However, in certain cases it is easy to construct
non-trivial Severi--Brauer surfaces with an action of the group~\mbox{$\mumu_3^2$}.

\begin{example}\label{example:cyclic-algebra}
Let $\KK$ be a field of characteristic different from $3$ that contains a primitive
cubic root of unity $\omega$. Let $a, b\in \KK$ be the elements such that $b$ is not
a cube in $\KK$, and~$a$ is not contained in the image of the Galois norm for the
field extension~\mbox{$\KK\subset\KK(\sqrt[3]{b})$}.
Consider the algebra~$A$ over~$\KK$ generated by variables $u$ and $v$
subject to relations
$$
u^3=a,\quad v^3=b,\quad uv=\omega vu.
$$
Then $A$ is a central division algebra, see for instance~\mbox{\cite[Exercise~3.1.6(ii),(iv)]{GS}}.
One has $\dim A=9$, so that $A$ corresponds to a non-trivial Severi--Brauer
surface $S$. Conjugation by $u$ defines an automorphism of order $3$ of $A$
(sending $u$ to $u$ and $v$ to~$\omega v$). Together with conjugation by
$v$ it generates a group $\mumu_3^2$ acting by automorphisms of $A$ and $S$.
\end{example}

%\bibliography{SB}
%\bibliographystyle{alpha}

\end{document}